\documentclass[11pt,reqno]{amsart}
\usepackage{caption}
 \usepackage{amssymb,amsfonts,amscd,amsbsy, color, amsmath}
\usepackage[mathscr]{eucal}
\usepackage{url}
\usepackage{graphicx}
\usepackage{mathtools}
\usepackage{todonotes}
\usepackage{tabularx}
\usepackage{tabu}
\usepackage{float, color}
\usepackage{placeins}
\usepackage{tikz}
\usepackage{bm}

\newtheorem{theorem}{Theorem}[section]
\newtheorem{lemma}[theorem]{Lemma}
\newtheorem{proposition}[theorem]{Proposition}
\newtheorem{corollary}[theorem]{Corollary}

\newtheorem{conjecture}[theorem]{Conjecture}
\theoremstyle{definition}
\newtheorem{remark}[theorem]{Remark}
\numberwithin{equation}{section}

\newcommand{\ib}[1]{\bm{#1}}

\makeatletter
\def\imod#1{\allowbreak\mkern5mu({\operator@font mod}\,\,#1)}
\makeatother
\allowdisplaybreaks

\makeatletter
\@namedef{subjclassname@2020}{%
  \textup{2020} Mathematics Subject Classification}
\makeatother

\begin{document}

\title[On sporadic sequences]{On sporadic sequences}

\author{Brendan Alinquant}
\author{Robert Osburn}

\address{School of Mathematics and Statistics, University College Dublin, Belfield, Dublin 4, Ireland}
\address{Theoretical Sciences Visiting Program, Okinawa Institute of Science and Technology Graduate University, Onna, 904-0495, Japan}

\email{robert.osburn@ucd.ie}
\email{brendan.alinquant@ucdconnect.ie}


\date{\today}

\begin{abstract}
In this note, we prove the last remaining case of the original 15 two-term supercongruence conjectures for sporadic sequences. The proof utilizes a new representation for this sequence (due to Gorodetsky) as the constant term of powers of a Laurent polynomial.
\end{abstract}

\maketitle

\section{Introduction}

The original 15 sporadic sequences are integer solutions to specific three-term recurrences. The first six with labels $\bf{A}$, $\bf{B}$, $\bf{C}$, $\bf{D}$, $\bf{E}$ and $\bf{F}$ were found by Zagier \cite{zag}, the next six denoted $(\alpha)$, $(\gamma)$, $(\delta)$, $(\epsilon)$, $(\eta)$ and $(\zeta)$ were discovered by Almkvist and Zudilin \cite{az} while the final three $s_7$, $s_{10}$ and $s_{18}$ are due to Cooper \cite{cooper1}. The Ap{\'e}ry numbers for $\zeta(2)$ and $\zeta(3)$ are $\bf{D}$ and $(\gamma)$, respectively. These sequences are ``sporadic" in the sense that they are not terminating, polynomial, hypergeometric or Legendrian solutions \cite[Section 3]{zag}. Each of the 15 cases has a modular parametrization \cite[Tables 1--3]{cooper2}, binomial sum representation \cite[Tables 1 and 2]{oss} and geometric origin \cite{zag1}. One intriguing aspect of these sequences lies in their arithmetic properties. In \cite{os, oss}, the following two-term supercongruences were conjectured.

\begin{conjecture} \label{OSS} Let $A(n)$ be one of the 15 original sporadic sequences. Then, for all primes $p \geq 5$ and all integers $m$, $r \geq 1$,
\begin{equation*} \label{con}
A(mp^r) \equiv A(mp^{r-1}) \pmod{p^{\lambda r}}
\end{equation*}
where $\lambda = 3$ except in the cases $\bf{B}$, $\bf{C}$, $\bf{E}$, $\bf{F}$ and $s_{18}$ in which case $\lambda = 2$.
\end{conjecture}
Conjecture \ref{OSS} has been established via the techniques in \cite{coster} for $\bf{A}$, $\bf{D}$, $(\gamma)$ and $s_{10}$, \cite{os2} for $\bf{C}$, \cite{os3} for $\bf{E}$ and $(\alpha)$, \cite{oss} for $(\epsilon)$, $(\eta)$, $s_7$ and $s_{18}$, \cite{gor1} for $(\zeta)$, \cite{gor2} for $\bf{B}$ and \cite{straub} for $\bf{F}$. This leaves the case $(\delta)$ which can be expressed as \cite{az}
\begin{equation} \label{az}
A_{\delta}(n) = \sum_{k=0}^{\left \lfloor \frac{n}{3} \right \rfloor} (-1)^k 3^{n-3k} \binom{n}{3k} \binom{n+k}{n} \frac{(3k)!}{k!^3}
\end{equation}
where  $\lfloor \,\,\, \rfloor$ is the usual floor function. The appearance of powers of 3 in (\ref{az}) causes difficulty. For example, after considerable work the $r=1$ case of Conjecture \ref{OSS} for $(\delta)$ was confirmed in \cite{at}. Fortunately, there is an alternative representation for $A_{\delta}(n)$ as the constant term of $\Lambda(x,y,z)^n$ where
\begin{equation*} \label{gord}
\Lambda(x,y,z) = \frac{(x+y-1)(x+z+1)(y-x+z)(y-z+1)}{xyz}.
\end{equation*}
Let $\ib{a}=(a_1, \dotsc, a_{\ell})$ be a tuple of nonnegative integers such that $a_1 + \dotsc + a_{\ell} = n$ and consider the multinomial coefficient
\begin{equation*}
\binom{n}{\ib{a}} = \binom{n}{a_1, \dotsc, a_{\ell}} := \frac{n!}{a_1 ! \cdots a_{\ell} !}.
\end{equation*}
In \cite[Proposition 3.3]{gor2}, Gorodetsky proved that 
\begin{equation} \label{delta}
A_{\delta}(n) = \sum_{(\ib{a}, \ib{b}, \ib{c}, \ib{d}) \, \in \, U(n)} (-1)^{a_2 + b_1 + d_3} \binom{n}{\ib{a}} \binom{n}{\ib{b}} \binom{n}{\ib{c}} \binom{n}{\ib{d}}   
\end{equation}
where
\begin{equation} \label{Uset}
U(n) = \left \{ (\ib{a}, \ib{b}, \ib{c}, \ib{d}) \in \mathbb{Z}_{\geq 0}^{12} \; : \;  \begin{matrix} a_1 + a_2 + a_3 = n, & b_1 + c_1 + d_1 = n \\ b_1 + b_2 + b_3 = n, & a_1 + b_2 + d_2 = n \\ c_1 + c_2 + c_3 = n, & a_2 + b_3 + c_2 = n \\ d_1 + d_2 + d_3 = n, & a_3 + c_3 + d_3 = n \end{matrix}     \right \}.
\end{equation}
Note that the final condition in (\ref{Uset}) is the result of taking the sum of the first four equations and then applying the fifth, sixth and seventh equations. We have added this extra equation in order to streamline the subsequent discussion. The purpose of this note is to utilize (\ref{delta}) and (\ref{Uset}) to resolve Conjecture \ref{OSS} in this last remaining case. Our main result is the following.

\begin{theorem} \label{main} For all primes $p \geq 5$ and integers $m$, $r \geq 1$, we have
\begin{equation} \label{Adelta} 
A_{\delta}(mp^r) \equiv A_{\delta}(mp^{r-1}) \pmod{p^{3r}}.
\end{equation}
\end{theorem}

The paper is organized as follows. In Section 2, we present some preliminaries, including two key steps in the proof of Theorem \ref{main}. The first step (see Proposition \ref{first}) reduces the proof to considering only those tuples in (\ref{delta}) which are not divisible by $p$ while the second step (see Proposition \ref{UintoT}) decomposes a certain subset of $U(mp^r)$ into disjoint unions of sets. In Section 3, we prove Theorem \ref{main}. In Section 4, we briefly address two nice questions posed by the referee. We make three final remarks. First, Straub \cite{straub} used (\ref{delta}) and (\ref{Uset}) (without the final condition) to prove (\ref{con}) for the case $(\delta)$ with $\lambda = 2$. Second, Cooper \cite{cooper2} has recently investigated sequences that are solutions to four-term recurrences and exhibit several novel features. For example, some of the sequences take values in $\mathbb{Z}[i]$ or $\mathbb{Z}[\sqrt{2}]$ and appear to satisfy ``rare" supercongruences \cite[Conjectures 11.1--11.5]{cooper2}. Lastly, there is currently no general framework which explains either Conjecture \ref{OSS} or \cite[Conjecture 1.3]{mos}. For a recent promising development in this direction, see \cite{bv}.

\section{Preliminaries}

We first recall the following result \cite[Theorem 2.2]{gessel}.

\begin{lemma} \label{jg} For primes $p \geq 5$ and integers $m$, $k$ and $r$, $s \geq 1$, 

\begin{equation*}
\binom{m p^r}{k p^s} / \binom{m p^{r-1}}{k p^{s-1}} \equiv 1 \pmod{p^{r+s+ \text{min}(r,s)}}.
\end{equation*}
\end{lemma}

\begin{remark} \label{23} Lemma \ref{jg} also holds for primes $2$ and $3$ with the exponent replaced by $r+s + \text{min}(r,s) - 2$ and $r+s + \text{min}(r,s) - 1$, respectively.
\end{remark}

Note that if $p \nmid k$ and $s \leq r$, then

\begin{equation} \label{easier}
\binom{m p^r}{k p^s} = p^{r-s} \frac{m}{k} \binom{mp^r - 1}{k p^s - 1} \equiv 0 \pmod{p^{r-s}}.
\end{equation}

Let $\sum^{'}$ denote the sum over indices not divisible by $p$. 

\begin{lemma} \label{har} For primes $p \geq 3$ and integers $s \geq 0$,
\begin{equation*}
\sideset{}{'}\sum_{x=1}^{p^s - 1} \frac{(-1)^x}{x^2} \equiv 0 \pmod{p^s}.
\end{equation*}
\end{lemma}

\begin{proof}
The terms in the sum can be grouped into disjoint pairs $\{x, p^s - x\}$ as $1 \leq x \leq p^s - 1$ implies  $x \not \equiv p^s - x \pmod{p^s}$. As $p \nmid x$, each of these pairs contributes 

\begin{equation*}
\frac{(-1)^x}{x^2} + \frac{(-1)^{p^s - x}}{(p^s - x)^2} \equiv 0 \pmod{p^s}
\end{equation*}
and so the result follows.
\end{proof}

We now present the two key steps in the proof of Theorem \ref{main}. For $\beta \in \mathbb{R}$, we write $\beta \ib{a} = (\beta a_1, \dotsc, \beta a_{\ell})$. Given $(\ib{a}, \ib{b}, \ib{c}, \ib{d}) \in U(n)$, let

\begin{equation} \label{Bdef}
B(\ib{a}, \ib{b}, \ib{c}, \ib{d}) := (-1)^{a_2 + b_1 + d_3} \binom{n}{\ib{a}} \binom{n}{\ib{b}} \binom{n}{\ib{c}} \binom{n}{\ib{d}}. 
\end{equation}
For a vector $\ib{a}$ with component $a_i$ and prime $p$, we write $p \mid \ib{a}$ to mean $p \mid a_i$ for every $i$ and $p \nmid \ib{a}$ if $p \nmid  a_i$ for some $i$. Moreover, we write $p \mid (\ib{a}, \ib{b}, \ib{c}, \ib{d})$ if $p \mid \ib{a}$, $\ib{b}$, $\ib{c}$ and $\ib{d}$ and $p \nmid (\ib{a}, \ib{b}, \ib{c}, \ib{d})$ if $p$ does not divide at least one of $\ib{a}$, $\ib{b}$, $\ib{c}$, $\ib{d}$. Finally, we define $v_{p}(\ib{a}) = \text{min} \{ v_{p}(a_i) : i \}$. We thank Armin Straub for his permission to include the proof of the following result (cf. \cite[Section 3]{straub}).

\begin{proposition} \label{first} For all primes $p \geq 5$ and integers $m$, $r \geq 1$ and $(\ib{a}, \ib{b}, \ib{c}, \ib{d}) \in U(mp^r)$ with $p \mid (\ib{a}, \ib{b}, \ib{c}, \ib{d})$, 

\begin{equation} \label{dbyp}
B(\ib{a}, \ib{b}, \ib{c}, \ib{d}) \equiv B((\ib{a}, \ib{b}, \ib{c}, \ib{d})/p) \pmod{p^{3r}}.
\end{equation}
\end{proposition}

\begin{proof}
Write $s_1 = \text{min} \, (\nu_{p}(a_1), r)$, $s_2 = \text{min} \, (\nu_{p}(a_2), r)$ and $s_3 = \text{min} \, (\nu_{p}(a_3), r)$ and suppose $s_1 \geq s_2$ and $s_3$. Thus, $a_3 = mp^r - a_1 - a_2$ is divisible by exactly $p^{s_2}$, and hence $s_2 = s_3$. It follows from two applications of Lemma \ref{jg} applied to
\begin{equation*}
\binom{mp^r}{\ib{a}} = \binom{mp^r}{a_1} \binom{mp^r - a_1}{a_2}
\end{equation*}
(note that $p^{s_1}$ divides $mp^r - a_1$) that since $p \mid \ib{a}$
\begin{equation} \label{ratio0}
\binom{mp^r}{\ib{a}} / \binom{mp^{r-1}}{\ib{a}/p} \equiv 1 \pmod{p^{s_1 + 2s_2}}.
\end{equation}
The same argument applies with $\ib{a}$ replaced by $\ib{b}$, $\ib{c}$ or $\ib{d}$. Suppose that the value of the quantity $s_1 + 2s_2$ is smallest for $\ib{a}$ in comparison with those for $\ib{b}$, $\ib{c}$ and $\ib{d}$. Then, using (\ref{ratio0}) leads to
\begin{equation} \label{ratio}
\frac{B(\ib{a}, \ib{b}, \ib{c}, \ib{d})}{B((\ib{a}, \ib{b}, \ib{c}, \ib{d})/p)} \equiv 1 \pmod{p^{s_1 + 2s_2}}.
\end{equation}
By the constraints defining $U(n)$ in (\ref{Uset}), we may write
\begin{equation} \label{form}
 B(\ib{a}, \ib{b}, \ib{c}, \ib{d}) = (-1)^{a_2 + b_1 + d_3} \binom{mp^r}{a_1, b_2, d_2} \binom{mp^r}{a_2, b_3, c_2} \binom{mp^r}{a_3, c_3, d_3} \binom{mp^r}{b_1, c_1, d_1}.
\end{equation}
In particular, 
\begin{equation} \label{three}
\binom{mp^r}{a_1} \binom{mp^r}{a_2} \binom{mp^r}{a_3} \mid B(\ib{a}, \ib{b}, \ib{c}, \ib{d}).
\end{equation}
For $i=1$, $2$ or $3$, write $a_i = k_i p^{s_i}$. Recall $s_i = \text{min} \, (\nu_{p}(a_i), r)$. There are now two cases. If $r \leq \nu_{p}(a_i)$, then $s_i=r$. Otherwise, $s_i = \nu_{p}(a_i) \leq r$ and $p \nmid k_i$. In either case, we may apply (\ref{easier}) and (\ref{three}) to obtain
\begin{equation} \label{last0}
B(\ib{a}, \ib{b}, \ib{c}, \ib{d}) \equiv 0 \pmod{p^{3r-s_1-2s_2}}.
\end{equation}
Combining (\ref{ratio}) and (\ref{last0}) yields (\ref{dbyp}).
\end{proof}

For a fixed prime $p \geq 5$ and integers $m$, $r \geq 1$, consider the sets
\begin{equation} \label{Uab}
U_{\ib{a} \ib{b}}(mp^r) := \{ (\ib{a}, \ib{b}, \ib{c}, \ib{d})  \in U(mp^r) \, : \, p \nmid \ib{a} \text{ and } \ib{b}, \, p \mid \ib{c} \text{ and } \ib{d} \}
\end{equation}
and 
\begin{equation*} \label{Us}
U_{\ib{a} \ib{b}}^{(s)}(mp^r) := \{ (\ib{a}, \ib{b}, \ib{c}, \ib{d}) \in U_{\ib{a} \ib{b}}(mp^r) \, : \,s =\,\text{min} \, (\nu_{p}(\ib{c}), \nu_{p}(\ib{d}), r) \}.
\end{equation*}
Clearly, $1 \leq s \leq r$. Furthermore, let $\bm{x} := (x, -x, 0, 0, -x, x, 0, 0, 0, 0, 0, 0) \in \mathbb{Z}^{12}$, 
\begin{equation*}
S_{p^s} := \{ 0 \leq x \leq p^s \, : \, p \nmid x \}
\end{equation*}
and 
\begin{equation*}
L_{s}(mp^r) := \left\{ \bm{\ell} \in p^s \, \mathbb{Z}_{\geq 0}^{12} \, : \, \bm{\ell} + \bm{x} \in U_{\ib{a} \ib{b}}^{(s)}(mp^r) \, \text{for some $x \in S_{p^s}$} \right \}.
\end{equation*}
Note that the sets $L_{s}(mp^r)$ are disjoint as $\bm{\ell} \in L_{s}(mp^r)$ implies $\nu_{p}(\bm{\ell})=s$. Finally, for each $\bm{\ell} \in L_{s}(mp^r)$, consider
\begin{equation*}
T_{s, \ib{\ell}} := \left \{ \bm{\alpha} \in \mathbb{Z}^{12} \, : \, \bm{\alpha} = \bm{\ell} + \bm{x} \,\, \text{for} \,\, x \in S_{p^s} \right \}.
\end{equation*}

\begin{proposition} \label{UintoT} We have
\begin{equation*}
U_{\ib{a} \ib{b}}(mp^r) = \bigsqcup_{1 \leq s \leq r} \, \bigsqcup_{\ib{\ell} \, \in \, L_s(mp^r)} T_{s,\bm{\ell}}.
\end{equation*}
\end{proposition}

\begin{proof} 
We first observe that the sets $T_{s, \ib{\ell}}$ are disjoint for each fixed $s$. Let  $(\ib{a}, \ib{b}, \ib{c}, \ib{d}) \in U_{\ib{a} \ib{b}}(mp^r)$. Then  $(\ib{a}, \ib{b}, \ib{c}, \ib{d}) \in U_{\ib{a} \ib{b}}^{(s)}(mp^r)$ for $s =\,\text{min} \, (\nu_{p}(\ib{c}), \nu_{p}(\ib{d}), r)$. Reducing the fifth and eighth equation in (\ref{Uset}), we find $p^s \mid b_1$ and $p^s \mid a_3$, respectively. Hence, we may write
\begin{equation*}
(\ib{a}, \ib{b}, \ib{c}, \ib{d}) = p^s \bm{k} + (x, y, 0, 0, z, w, 0, 0, 0, 0, 0, 0, 0)
\end{equation*}
for some $\bm{k} \in \mathbb{Z}_{\geq 0}^{12}$ and $x$, $y$, $z$, $w \in S_{p^s}$. Now, reducing the first, sixth and seventh equations modulo $p^s$ in (\ref{Uset}), we obtain $-y \equiv -z \equiv w \equiv x \pmod{p^s}$ and so
\begin{equation*}
(\ib{a}, \ib{b}, \ib{c}, \ib{d}) = p^s \bm{k'} + \bm{x} \in U_{\ib{a} \ib{b}}^{(s)}(mp^r)
\end{equation*}
where only the components $k_2$, $k_5$ and $k_6$ of $\bm{k}$ have been redefined to form $\bm{k'}$. Thus, $\bm{\ell} := p^s \bm{k'} \in L_s(mp^r)$ and so $(\ib{a}, \ib{b}, \ib{c}, \ib{d}) \in T_{s, \bm{\ell}}$. Conversely, let $\bm{\alpha} \in T_{s, \bm{\ell}}$ for some $1 \leq s \leq r$ and $\bm{\ell} \in L_{s}(mp^r)$. Then $\bm{\alpha} = \bm{\ell} + \bm{x} = p^s \bm{k} + \bm{x} \in U_{\ib{a} \ib{b}}^{(s)}(mp^r)$ for some $\bm{k} \in \mathbb{Z}_{\geq 0}^{12}$ and $x \in S_{p^s}$. Note that, in fact, if we replace the component $x$ in $\bm{x}$ with any $y \in S_{p^s}$ and write $\bm{\alpha'}$ for the associated vector, then the eight equations in (\ref{Uset}) are still satisfied for $\bm{\alpha'}$. Also, since $x$, $y \in S_{p^s}$, $p^s k_2 - x \geq 0$ and $p^s k_5 - x \geq 0$, we have $p^s k_2 - y \geq 0$ and $p^s k_5 - y \geq 0$. Hence, $\bm{\alpha'} \in U_{\ib{a} \ib{b}}^{(s)}(mp^r)$ and so $T_{s, \bm{\ell}} \subseteq U_{\ib{a} \ib{b}}^{(s)}(mp^r)$.
\end{proof}

Let $\lceil \,\,\, \rceil$ denote the usual ceiling function. The next two results follow from splitting the defined product of the binomial coefficient, according to whether the index is divisible by $p$ or not. The first result is \cite[Lemma 2.4]{oss}. We omit the proof of the second result.

\begin{lemma} \label{l1} For primes $p$, integers $m$ and integers $k \geq 0$, $s \geq 1$,
\begin{equation*}
\binom{mp^s - 1}{k} (-1)^k \equiv \binom{mp^{s-1} - 1}{\lfloor k/p \rfloor} (-1)^{\lfloor k/p \rfloor} \pmod{p^s}.
\end{equation*}
\end{lemma}

\begin{lemma} \label{l2} For primes $p$, integers $m$ and integers $k$, $\ell \geq 0$, $s \geq 1$,
\begin{equation*}
\binom{mp^s - \ell}{kp^s} \equiv \binom{mp^{s-1} - \lceil \ell/p \rceil}{kp^{s-1}} \pmod{p^s}.
\end{equation*}
\end{lemma}

For tuples $\ib{\ell}:=({\ell}_1, \dotsc, {\ell}_{12}) \in \mathbb{Z}^{12}$, integers $n$ and $x \in \mathbb{R}$, define
\begin{equation} \label{choose}
C(\ib{\ell}, n, x) = \binom{n-1}{{\ell}_1 + \lfloor x \rfloor} \binom{n - {\ell}_1 - \lceil x \rceil}{{\ell}_3} \binom{n-1}{{\ell}_6 - \lfloor x \rfloor} \binom{n - {\ell}_6 - \lceil x \rceil}{{\ell}_4}.
\end{equation}
The reason for the choice of (\ref{choose}) will be made clear in the proof of Theorem \ref{main}, in particular see (\ref{reduce}).

\begin{lemma} \label{l3} Let $p$ be an odd prime such that $p^s$ divides both $\ib{\ell}$ and $n$ for some $s \geq 1$ and $x \in \mathbb{R}$ with $0 \leq x \leq p^s$. Then
\begin{equation*}
C(\ib{\ell}, n, x) \equiv C(\ib{\ell}/p, n/p, x/p) \pmod{p^s}.
\end{equation*}
\end{lemma}

\begin{proof}
Let $\ib{\ell} = (p^s k_1, \dotsc, p^s k_{12})$ and $n=mp^s$ where $k_i$, $m \in \mathbb{Z}$ for $1 \leq i \leq 12$. After applying Lemmas \ref{l1} and \ref{l2}, using the facts $\lfloor x/p \rfloor = \left \lfloor \lfloor x \rfloor /p \right \rfloor$ and $\lceil x/p \rceil = \left \lceil \lceil x \rceil / p \right \rceil$ and simplifying, we obtain
\begin{equation*}
\begin{aligned}
C(\ib{\ell}, mp^s, x) & = \binom{mp^s - 1}{p^s k_1 + \lfloor x \rfloor} \binom{mp^s - p^s k_1 - \lceil x \rceil}{p^s k_3} \binom{mp^s - 1}{p^s k_6 + \lfloor x \rfloor} \binom{mp^s - p^s k_6 - \lceil x \rceil}{p^s k_4} \\
& \equiv \binom{mp^{s-1} - 1}{p^{s-1} k_1 + \lfloor x/p \rfloor} \binom{mp^{s-1} - p^{s-1} k_1 - \lceil x/p \rceil}{p^{s-1} k_3}  \binom{mp^{s-1} - 1}{p^{s-1} k_6 + \lfloor x/p \rfloor} \\
&\quad \quad \times \binom{mp^{s-1} - p^{s-1} k_6 - \lceil x/p \rceil}{p^{s-1} k_4} \pmod{p^s} \\
& \equiv C(\ib{\ell}/p, mp^{s-1}, x/p) \pmod{p^s}.
\end{aligned}
\end{equation*}

\end{proof}

\section{Proof of Theorem \ref{main}}

We are now in a position to prove our main result.

\begin{proof}[Proof of Theorem \ref{main}]
We begin by splitting the sum in (\ref{delta}) as 
\begin{equation} \label{split}
\begin{aligned}
A_{\delta}(mp^r) & = \sum_{(\ib{a}, \ib{b}, \ib{c}, \ib{d}) \, \in \, U(mp^r)} B(\ib{a}, \ib{b}, \ib{c}, \ib{d}) \\
& = \sum_{\substack{(\ib{a}, \ib{b}, \ib{c}, \ib{d}) \, \in \, U(mp^r) \\ p \, \mid \, (\ib{a}, \ib{b}, \ib{c}, \ib{d})}} B(\ib{a}, \ib{b}, \ib{c}, \ib{d}) \,\, +  \sum_{\substack{(\ib{a}, \ib{b}, \ib{c}, \ib{d}) \, \in \, U(mp^r) \\ p \, \nmid \, (\ib{a}, \ib{b}, \ib{c}, \ib{d})}} B(\ib{a}, \ib{b}, \ib{c}, \ib{d}).
\end{aligned}
\end{equation}
By Proposition \ref{first} and the fact that $(\ib{a}, \ib{b}, \ib{c}, \ib{d}) \in U(mp^{r-1})$ if and only if $p (\ib{a}, \ib{b}, \ib{c}, \ib{d}) \in U(mp^{r})$,
\begin{equation} \label{det}
\begin{aligned}
\sum_{\substack{(\ib{a}, \ib{b}, \ib{c}, \ib{d}) \, \in \, U(mp^r) \\ p \, \mid \, (\ib{a}, \ib{b}, \ib{c}, \ib{d})}} B(\ib{a}, \ib{b}, \ib{c}, \ib{d}) \,\, & \equiv \sum_{\substack{(\ib{a}, \ib{b}, \ib{c}, \ib{d}) \, \in \, U(mp^r) \\ p \, \mid \, (\ib{a}, \ib{b}, \ib{c}, \ib{d})}}  B((\ib{a}, \ib{b}, \ib{c}, \ib{d})/p)  \pmod{p^{3r}} \\
& \equiv \sum_{(\ib{a}, \ib{b}, \ib{c}, \ib{d}) \, \in \, U(mp^{r-1})} B(\ib{a}, \ib{b}, \ib{c}, \ib{d}) \pmod{p^{3r}} \\
& \equiv A_{\delta}(mp^{r-1}) \pmod{p^{3r}}. 
\end{aligned}
\end{equation} 
By (\ref{split}) and (\ref{det}), it suffices to prove
\begin{equation} \label{suff}
\sum_{\substack{(\ib{a}, \ib{b}, \ib{c}, \ib{d}) \, \in \, U(mp^r) \\ p \, \nmid \, (\ib{a}, \ib{b}, \ib{c}, \ib{d})}} B(\ib{a}, \ib{b}, \ib{c}, \ib{d}) \equiv 0 \pmod{p^{3r}}.
\end{equation}
Note that when $p$ does not divide one of $\ib{a}$, $\ib{b}$, $\ib{c}$ or $\ib{d}$, then there must be at least one associated component which is not divisible by $p$. Choose one of the last four equations in (\ref{Uset}) such that this component appears and reduce it modulo $p$. One then finds that $p$ also does not divide at least one of the other two components in this equation. Thus, $p$ does not divide at least two of $\ib{a}$, $\ib{b}$, $\ib{c}$, $\ib{d}$ and so the left-hand side of (\ref{suff}) becomes
\begin{equation} \label{threesums}
\begin{aligned}
\sum_{\substack{(\ib{a}, \ib{b}, \ib{c}, \ib{d}) \, \in \, U(mp^r) \\ p \, \nmid \, (\ib{a}, \ib{b}, \ib{c}, \ib{d}) }} B(\ib{a}, \ib{b}, \ib{c}, \ib{d})  &= \sum_{\substack{(\ib{a}, \ib{b}, \ib{c}, \ib{d}) \, \in \, U(mp^r) \\ p \, \nmid \, \ib{a}, \ib{b}, \ib{c} \,\, \text{and} \,\, \ib{d}}}  B(\ib{a}, \ib{b}, \ib{c}, \ib{d}) \,\,  + \sum_{\substack{(\ib{a}, \ib{b}, \ib{c}, \ib{d}) \, \in \, U(mp^r) \\ p \, \nmid \, \text{exactly 3 of } \ib{a}, \ib{b}, \ib{c}, \ib{d}}}  B(\ib{a}, \ib{b}, \ib{c}, \ib{d})  \\
& + \sum_{\substack{(\ib{a}, \ib{b}, \ib{c}, \ib{d}) \, \in \, U(mp^r) \\ p \, \nmid \, \text{exactly 2 of } \ib{a}, \ib{b}, \ib{c},\ib{d}}}  B(\ib{a}, \ib{b}, \ib{c}, \ib{d}).
\end{aligned}
\end{equation}
We first claim that
\begin{equation} \label{claim}
\sum_{\substack{(\ib{a}, \ib{b}, \ib{c}, \ib{d}) \, \in \, U(mp^r) \\ p \, \nmid \, \ib{a}, \ib{b}, \ib{c} \,\, \text{and} \,\, \ib{d}}}  B(\ib{a}, \ib{b}, \ib{c}, \ib{d})   \,\, \equiv \sum_{\substack{(\ib{a}, \ib{b}, \ib{c}, \ib{d}) \, \in \, U(mp^r) \\ p \, \nmid \, \text{exactly 3 of } \ib{a}, \ib{b}, \ib{c}, \ib{d}}} B(\ib{a}, \ib{b}, \ib{c}, \ib{d}) \equiv 0 \pmod{p^{3r}}.
\end{equation}
To see (\ref{claim}), we observe that if $p$ does not divide $\ib{m} \in \{\ib{a}, \ib{b}, \ib{c}, \ib{d} \}$, then $p$ does not divide an associated component, say, $m_i$. So, 
\begin{equation} \label{twoterms}
\binom{mp^r}{\ib{m}} = \binom{mp^r}{m_i} \binom{mp^r - m_i}{m_j}
\end{equation}
where $j \neq i$. Then by (\ref{easier}), the right-hand side of (\ref{twoterms}) is divisible by $p^r$. Hence,  if $p$ does not divide $\ib{a}$, $\ib{b}$, $\ib{c}$ and $\ib{d}$ or $p$ does not divide three of $\ib{a}$, $\ib{b}$, $\ib{c}$, $\ib{d}$, then $p^{3r}$ divides $B(\ib{a}, \ib{b}, \ib{c}, \ib{d})$ and thus (\ref{claim}) follows. 

Now, consider the set $U_{\ib{a} \ib{b}}(mp^r)$ given by (\ref{Uab}). The sets $U_{\ib{a} \ib{c}}(mp^r)$, $U_{\ib{a} \ib{d}}(mp^r)$, $U_{\ib{b} \ib{c}}(mp^r)$, $U_{\ib{b} \ib{d}}(mp^r)$ and $U_{\ib{c} \ib{d}}(mp^r)$ are similarly defined. We clearly have
\begin{equation} \label{one2six}
	\begin{aligned}
		\sum_{\substack{(\ib{a}, \ib{b}, \ib{c}, \ib{d}) \, \in \, U(mp^r) \\ p \, \nmid \, \text{exactly 2 of } \ib{a}, \ib{b}, \ib{c},\ib{d}}}  B(\ib{a}, \ib{b}, \ib{c}, \ib{d}) & = \sum_{(\ib{a}, \ib{b}, \ib{c}, \ib{d}) \, \in \, U_{\ib{a} \ib{b}}(mp^r)} B(\ib{a}, \ib{b}, \ib{c}, \ib{d}) \,\, + \sum_{(\ib{a}, \ib{b}, \ib{c}, \ib{d}) \, \in \, U_{\ib{a} \ib{c}}(mp^r)} B(\ib{a}, \ib{b}, \ib{c}, \ib{d}) \\
		& + \sum_{(\ib{a}, \ib{b}, \ib{c}, \ib{d}) \, \in \, U_{\ib{a} \ib{d}}(mp^r)} B(\ib{a}, \ib{b}, \ib{c}, \ib{d}) \,\, + \sum_{(\ib{a}, \ib{b}, \ib{c}, \ib{d}) \, \in \, U_{\ib{b} \ib{c}}(mp^r)} B(\ib{a}, \ib{b}, \ib{c}, \ib{d}) \\
		& + \sum_{(\ib{a}, \ib{b}, \ib{c}, \ib{d}) \, \in \, U_{\ib{b} \ib{d}}(mp^r)} B(\ib{a}, \ib{b}, \ib{c}, \ib{d}) \\
		& + \sum_{(\ib{a}, \ib{b}, \ib{c}, \ib{d}) \, \in \, U_{\ib{c} \ib{d}}(mp^r)} B(\ib{a}, \ib{b}, \ib{c}, \ib{d}).
	\end{aligned}
\end{equation}
Our second claim is that each of the sums on the right-hand side of (\ref{one2six}) is equal. To deduce this, consider the maps on $U(mp^r)$ given by
\begin{align}
(a_1, a_2, a_3, b_1, b_2, b_3, c_1, c_2, c_3, d_1, d_2, d_3) & \mapsto (d_2, d_3, d_1, a_2, a_1, a_3, c_2, c_3, c_1, b_3, b_2, b_1), \label{b1} \\
(a_1, a_2, a_3, b_1, b_2, b_3, c_1, c_2, c_3, d_1, d_2, d_3) & \mapsto (b_2, b_1, b_3, d_3, d_2, d_1, c_3, c_1, c_2, a_3, a_1, a_2), \label{b2} \\
(a_1, a_2, a_3, b_1, b_2, b_3, c_1, c_2, c_3, d_1, d_2, d_3) & \mapsto (a_3, a_2, a_1, c_1, c_3, c_2, b_1, b_3, b_2, d_1, d_3, d_2). \label{b3}
\end{align}
Applying (\ref{b1}) or (\ref{b2}) does not change the sign or the product of the multinomial coefficients in (\ref{Bdef}) and so if $(\ib{a}, \ib{b}, \ib{c}, \ib{d}) \mapsto (\ib{e}, \ib{f}, \ib{g}, \ib{h})$, then $B(\ib{a}, \ib{b}, \ib{c}, \ib{d}) = B(\ib{e}, \ib{f}, \ib{g}, \ib{h})$. The map (\ref{b1}) yields bijections between $U_{\ib{a} \ib{b}}(mp^r)$ and $U_{\ib{a} \ib{d}}(mp^r)$ and $U_{\ib{a} \ib{c}}(mp^r)$ and $U_{\ib{c} \ib{d}}(mp^r)$. Similarly, the map (\ref{b2}) gives bijections between $U_{\ib{a} \ib{d}}(mp^r)$ and $U_{\ib{b} \ib{d}}(mp^r)$ and $U_{\ib{c} \ib{d}}(mp^r)$ and $U_{\ib{b} \ib{c}}(mp^r)$. Applying (\ref{b3}), one can check that (\ref{Bdef}) is unchanged and we have a bijection between $U_{\ib{a} \ib{b}}(mp^r)$ and $U_{\ib{a} \ib{c}}(mp^r)$. Thus, the second claim follows and so by (\ref{threesums}), (\ref{claim}) and (\ref{one2six}), we have
\begin{equation*} \label{remain}
\sum_{\substack{(\ib{a}, \ib{b}, \ib{c}, \ib{d}) \, \in \, U(mp^r) \\ p \, \nmid \, (\ib{a}, \ib{b}, \ib{c}, \ib{d})}} B(\ib{a}, \ib{b}, \ib{c}, \ib{d}) \equiv 6 \left (\sum_{\substack{(\ib{a}, \ib{b}, \ib{c}, \ib{d}) \, \in \, U_{\ib{a} \ib{b}}(mp^r)}} B(\ib{a}, \ib{b}, \ib{c}, \ib{d}) \right) \pmod{p^{3r}}.
\end{equation*}

By Proposition \ref{UintoT},
\begin{equation} \label{rewrite}
\sum_{\substack{(\ib{a}, \ib{b}, \ib{c}, \ib{d}) \, \in \, U_{\ib{a} \ib{b}}(mp^r)}} B(\ib{a}, \ib{b}, \ib{c}, \ib{d}) = \sum_{1 \leq s \leq r} \, \sum_{\bm{\ell} \, \in \,  L_s(mp^r)} \, \sum_{(\ib{a}, \ib{b}, \ib{c}, \ib{d}) \, \in \, T_{s, \bm{\ell}}}  B(\ib{a}, \ib{b}, \ib{c}, \ib{d}).
\end{equation}
Focusing on the inner sum in (\ref{rewrite}) yields
\begin{equation} \label{reduce}
\begin{aligned}
\sum_{(\ib{a}, \ib{b}, \ib{c}, \ib{d}) \, \in \, T_{s, \ib{\ell}}} B(\ib{a}, \ib{b}, \ib{c}, \ib{d}) & = \sum_{(\ib{a}, \ib{b}, \ib{c}, \ib{d}) \, \in \, T_{s, \ib{\ell}}} (-1)^{a_2 + b_1 + d_3} \binom{mp^r}{\ib{a}} \binom{mp^r}{\ib{b}} \binom{mp^r}{\ib{c}} \binom{mp^r}{\ib{d}} \\
& = \sideset{}{'}\sum_{x=1}^{p^s - 1} (-1)^{p^s k_2 - x + p^s k_4 + p^s k_{12}} \binom{mp^r}{\ib{a}} \binom{mp^r}{\ib{b}} \binom{mp^r}{\ib{c}} \binom{mp^r}{\ib{d}} \\
& = (-1)^{p^s k_{2} + p^s k_4 + p^s k_{12}} \binom{mp^r}{\ib{c}} \binom{mp^r}{\ib{d}} \sideset{}{'} \sum_{x=1}^{p^s - 1} (-1)^x \binom{mp^r}{\ib{a}} \binom{mp^r}{\ib{b}} \\
& = (-1)^{p^s k_{2} + p^s k_4 + p^s k_{12}} m^2 p^{2r} \binom{mp^r}{\ib{c}} \binom{mp^r}{\ib{d}} \\
&\quad \quad \times \sideset{}{'} \sum_{x=1}^{p^s - 1} \frac{(-1)^x}{a_1 b_3} \binom{mp^r-1}{a_1 - 1} \binom{mp^r - a_1}{a_3} \binom{mp^r - 1}{b_3 - 1} \binom{mp^r - b_3}{b_1}
\end{aligned}
\end{equation}
where $(\ib{a}, \ib{b}, \ib{c}, \ib{d}) = \bm{\ell} + x =: p^s \bm{k} + \bm{x}$ and the last step in (\ref{reduce}) follows from two instances of (\ref{easier}). Now, since $s = \text{min} \, (\nu_{p}(\ib{c}), \nu_{p}(\ib{d}), r)$, we find $p^{r-s} \mid \binom{mp^r}{\ib{c}} \binom{mp^r}{\ib{d}}$. Hence by (\ref{reduce}) it suffices to show that
\begin{equation} \label{reduce2}
\begin{aligned}
&\sideset{}{'}\sum_{x=1}^{p^s - 1} \frac{(-1)^x}{a_1 b_3} \binom{mp^r-1}{a_1 - 1} \binom{mp^r - a_1}{a_3} \binom{mp^r - 1}{b_3 - 1} \binom{mp^r - b_3}{b_1} \\
& \equiv \sideset{}{'}\sum_{x=1}^{p^s - 1} \frac{(-1)^x}{(p^s k_1 + x)(p^s k_6 + x)} \binom{mp^r - 1}{p^s k_1 + x - 1} \binom{mp^r - p^s k_1 - x}{p^s k_3} \binom{mp^r - 1}{p^s k_6 + x-1} \\
& \qquad \qquad \qquad \qquad \qquad \qquad \qquad \qquad \qquad \qquad \qquad \qquad \qquad \times \binom{mp^r - p^s k_6 - x}{p^s k_4} \pmod{p^s} \\  
& \equiv \sideset{}{'}\sum_{x=1}^{p^s - 1} \frac{(-1)^x}{x^2} \binom{mp^r - 1}{p^s k_1 + x - 1} \binom{mp^r - p^s k_1 - x}{p^s k_3} \binom{mp^r - 1}{p^s k_6 + x-1} \binom{mp^r - p^s k_6 - x}{p^s k_4} \pmod{p^s} \\
& \equiv 0 \pmod{p^s}.   
\end{aligned}
\end{equation}

If we now apply Lemmas \ref{l1} and \ref{l2} to (\ref{reduce2}), then it suffices to show
\begin{equation} \label{laststep}
\sideset{}{'}\sum_{x=1}^{p^s - 1} \frac{(-1)^x}{x^2} C(\ib{\ell}/p, mp^{r-1}, x/p) \equiv 0 \pmod{p^s}.
\end{equation}
To deduce (\ref{laststep}), our final claim is that
\begin{equation} \label{keyreduce}
\sideset{}{'}\sum_{x=1}^{p^s - 1} \frac{(-1)^x}{x^2} C(\ib{\ell}/p, mp^{r-1}, x/p) \equiv \sideset{}{'}\sum_{x=1}^{p^s - 1} \frac{(-1)^x}{x^2} C(\ib{\ell}/p^{t}, mp^{r-t}, x/p^t) \pmod{p^s}
\end{equation}
for each $1 \leq t \leq s$. We proceed by induction on $t$. The case $t=1$ clearly holds. Suppose (\ref{keyreduce}) is true for $t < s$. If $\lfloor x/p^t \rfloor = \lfloor y/p^t \rfloor$ for some $x$, $y \in S_{p^s}$, then $\lceil x/p^t \rceil = \lceil y/p^t \rceil$ since $x/p^t$, $y/p^t \not \in \mathbb{Z}$ and so $C(\ib{\ell}/p^t, mp^{r-t}, x/p^t) = C(\ib{\ell}/p^t, mp^{r-t}, y/p^t)$. Hence if $\lfloor x/p^t \rfloor$ is fixed, $C(\ib{\ell}/p^t, mp^{r-t}, x/p^t)$ is constant. Also, each $x \in S_{p^s}$ can be written as $x = np^t + y$ where $0 \leq n \leq p^{s-t} - 1$ and $y \in S_{p^t}$. Thus, by Lemma \ref{l3} with $s$ replaced by $s-t$ and the fact that
\begin{equation*}
\sideset{}{'}\sum_{x=np^t + 1}^{np^t + p^t -1} \frac{(-1)^x}{x^2} \equiv (-1)^{np^t} \sideset{}{'}\sum_{x=1}^{p^t-1} \frac{(-1)^x}{x^2} \equiv 0 \pmod{p^t}
\end{equation*}
which follows from Lemma \ref{har} with $s$ replaced by $t$, we have
%\begin{equation*}
\begin{align}
\sideset{}{'}\sum_{x=1}^{p^s - 1} \frac{(-1)^x}{x^2} C(\ib{\ell}/p, mp^{r-1}, x/p) & \equiv \sideset{}{'}\sum_{x=1}^{p^s - 1} \frac{(-1)^x}{x^2} C(\ib{\ell}/p^{t}, mp^{r-t}, x/p^t) \pmod{p^s} \label{e1} \\
& \equiv \sideset{}{'}\sum_{n=1}^{p^{s-t} - 1} \left ( \sideset{}{'}\sum_{x=np^t + 1}^{np^t + p^t - 1} \frac{(-1)^x}{x^2} C(\ib{\ell}/p^{t}, mp^{r-t}, x/p^t) \right ) \pmod{p^s} \label{e2} \\
& \equiv \sum_{n=1}^{p^{s-t} - 1} C(\ib{\ell}/p^{t}, mp^{r-t}, x/p^t) \left (  \sideset{}{'}\sum_{x=np^t + 1}^{np^t + p^t - 1} \frac{(-1)^x}{x^2} \right ) \pmod{p^s} \label{e3} \\
& \equiv \sum_{n=1}^{p^{s-t} - 1} C(\ib{\ell}/p^{t+1}, mp^{r-t-1}, x/p^{t+1})  \nonumber \\
& \qquad \qquad \qquad \qquad \qquad \qquad \times \left (  \sideset{}{'}\sum_{x=np^t + 1}^{np^t + p^t - 1} \frac{(-1)^x}{x^2} \right ) \pmod{p^s} \label{e4} \\
& \equiv \sum_{n=1}^{p^{s-t} - 1}  \sideset{}{'}\sum_{x=np^t + 1}^{np^t + p^t - 1} \frac{(-1)^x}{x^2} C(\ib{\ell}/p^{t+1}, mp^{r-t-1}, x/p^{t+1}) \pmod{p^s} \label{e5} \\
& \equiv \sideset{}{'} \sum_{x=1}^{p^s - 1} \frac{(-1)^x}{x^2} C(\ib{\ell}/p^{t+1}, mp^{r-t-1}, x/p^{t+1}) \pmod{p^s}. \label{e6}
\end{align}
%\end{equation*}
Thus, (\ref{keyreduce}) follows by induction on $t$. If we now let $t=s$ in (\ref{keyreduce}), use the fact that if $\lfloor x/p^s \rfloor$ is fixed, $C(\ib{\ell}/p^s, mp^{r-s}, x/p^s)$ is constant and apply Lemma \ref{har}, then
%\begin{equation*}
\begin{align}
\sideset{}{'}\sum_{x=1}^{p^s - 1} \frac{(-1)^x}{x^2} C(\ib{\ell}/p, mp^{r-1}, x/p) & \equiv \sideset{}{'}\sum_{x=1}^{p^s - 1} \frac{(-1)^x}{x^2} C(\ib{\ell}/p^s, mp^{r-s}, x/p^s) \pmod{p^s} \label{e7} \\
& \equiv C(\ib{\ell}/p^s, mp^{r-s}, x/p^s) \left ( \sideset{}{'}\sum_{x=1}^{p^s - 1} \frac{(-1)^x}{x^2} \right) \pmod{p^s} \label{e8} \\
& \equiv 0 \pmod{p^s}. \label{e9}
\end{align}
%\end{equation*}
This proves (\ref{laststep}), (\ref{suff}) and thus (\ref{Adelta}). 
\end{proof}

\section{Two questions}

In this section, we briefly address the following two nice questions asked by the referee: are there versions of Theorem \ref{main} for the primes $2$ and $3$ and for an unsigned analogue of (\ref{delta})? Consider the sequence

\begin{equation} \label{nosigndelta}
A^{\prime}_{\delta}(n) =  \sum_{(\ib{a}, \ib{b}, \ib{c}, \ib{d}) \, \in \, U(n)} B^{\prime}(\ib{a}, \ib{b}, \ib{c}, \ib{d}) \, \, := \sum_{(\ib{a}, \ib{b}, \ib{c}, \ib{d}) \, \in \, U(n)} \binom{n}{\ib{a}} \binom{n}{\ib{b}} \binom{n}{\ib{c}} \binom{n}{\ib{d}}   
\end{equation}
where $U(n)$ is given by (\ref{Uset}). The first few terms in this sequence (which do not yet appear in the OEIS) are $1$, $9$, $297$, $13833$, $748521$, $44127009$, $2750141241$, $\dotsc$. 

\begin{theorem} \label{main1} For all primes $p \geq 3$ and integers $m$, $r \geq 1$, we have
\begin{equation} \label{Aprime}
A^{\prime}_{\delta}(mp^r) \equiv A^{\prime}_{\delta}(mp^{r-1}) \pmod{p^{3r}}.
\end{equation}
For $p=2$, $(\ref{Aprime})$ holds modulo $2^{2r}$. Moreover, for $p=3$, $(\ref{Adelta})$ holds and for $p=2$, $(\ref{Adelta})$ holds modulo $2^{2r}$ for $r \geq 2$ and modulo 2 otherwise. 
\end{theorem}

In order to prove Theorem \ref{main1}, we require some setup. The first result follows by induction on $n$ while the second result is well-known.

\begin{lemma} \label{div} For all integers $a \geq 1$ we have 
\begin{equation*}
n! \mid \binom{na}{\, \underbrace{a, a, \dotsc, a}_{\text{$n$ copies}} \, }.
\end{equation*}
\end{lemma}

\begin{lemma} \label{has} For primes $p \geq 5$ and integers $s \geq 0$,
\begin{equation*}
\sideset{}{'}\sum_{x=1}^{p^s - 1} \frac{1}{x^2} \equiv 0 \pmod{p^s}
\end{equation*}
while for $p = 3$,
\begin{equation*}
\sideset{}{'}\sum_{x=1}^{3^s - 1} \frac{3}{x^2} \equiv 0 \pmod{3^s}.
\end{equation*}
\end{lemma}

The next step is to obtain a version of Proposition \ref{first} for $B^{\prime}(\ib{a}, \ib{b}, \ib{c}, \ib{d})$. 

\begin{proposition} \label{second} For all primes $p \geq 2$ and integers $m$, $r \geq 1$ and $(\ib{a}, \ib{b}, \ib{c}, \ib{d}) \in U(mp^r)$ with $p \mid (\ib{a}, \ib{b}, \ib{c}, \ib{d})$, 

\begin{equation} \label{dbyp1}
B^{\prime}(\ib{a}, \ib{b}, \ib{c}, \ib{d}) \equiv B^{\prime}((\ib{a}, \ib{b}, \ib{c}, \ib{d})/p) \pmod{p^{3r - \epsilon}}
\end{equation}
where $\epsilon = 2$, $1$ or $0$ for $p=2$, $p=3$ and primes $p \geq 5$, respectively.
\end{proposition}

\begin{proof}
The proof is the same as that for Proposition \ref{first} except one applies Remark \ref{23} twice and $p^{s_1 + 2s_2}$ is replaced by $p^{s_1 + 2s_2 - \epsilon}$ in (\ref{ratio0}) and (\ref{ratio}).
\end{proof}

\begin{corollary} \label{secondB}
For all primes $p \geq 3$ and integers $m$, $r \geq 1$ and $(\ib{a}, \ib{b}, \ib{c}, \ib{d}) \in U(mp^r)$ with $p \mid (\ib{a}, \ib{b}, \ib{c}, \ib{d})$,
\begin{equation} \label{dbyp2}
B(\ib{a}, \ib{b}, \ib{c}, \ib{d}) \equiv B((\ib{a}, \ib{b}, \ib{c}, \ib{d})/p) \pmod{p^{3r-\epsilon}}
\end{equation}
where $\epsilon = 1$ or $0$ for $p=3$ and primes $p \geq 5$, respectively. For $p = 2$, $(\ref{dbyp2})$ holds modulo $p^{2r-1}$.  
\end{corollary}

\begin{proof}
For primes $p \geq 3$, the signs in $B(\ib{a}, \ib{b}, \ib{c}, \ib{d})$ and $B((\ib{a}, \ib{b}, \ib{c}, \ib{d})/p)$ are equal as the associated exponents have the same parity. Thus, (\ref{dbyp2}) is equivalent to (\ref{dbyp1}). For $p=2$, there are two cases. If $4 \mid (\ib{a}, \ib{b}, \ib{c}, \ib{d})$, then the exponents of $B(\ib{a}, \ib{b}, \ib{c}, \ib{d})$ and $B((\ib{a}, \ib{b}, \ib{c}, \ib{d})/2)$ are both even and the result follows from (\ref{dbyp1}). If $2$ exactly divides $(\ib{a}, \ib{b}, \ib{c}, \ib{d})$, then one observes that $4$ does not divide at least two of $\ib{a}$, $\ib{b}$, $\ib{c}$, $\ib{d}$. It then follows from (\ref{easier}) that $B(\ib{a}, \ib{b}, \ib{c}, \ib{d})$ and $B((\ib{a}, \ib{b}, \ib{c}, \ib{d})/2)$ are both congruent to $0$ modulo $2^{2r-2}$. Taking the signs of $B(\ib{a}, \ib{b}, \ib{c}, \ib{d})$ and $B((\ib{a}, \ib{b}, \ib{c}, \ib{d})/2)$ into account modulo 2 and applying (\ref{dbyp1}) yields
\begin{equation} \label{modpower2}
B(\ib{a}, \ib{b}, \ib{c}, \ib{d}) \equiv B^{\prime}(\ib{a}, \ib{b}, \ib{c}, \ib{d}) \equiv B^{\prime}((\ib{a}, \ib{b}, \ib{c}, \ib{d})/2) \equiv B((\ib{a}, \ib{b}, \ib{c}, \ib{d})/2) \pmod{2^{2r-1}}.
\end{equation}
\end{proof}

For the vector
\begin{equation*}
(\ib{a}, \ib{b}, \ib{c}, \ib{d})  = (a_1, a_2, a_3, b_1, b_2, b_3, c_1, c_2, c_3, d_1, d_2, d_3) =: (k_1, \dotsc, k_{12})
\end{equation*}
and a permutation $\sigma = (\sigma(1), \dotsc, \sigma(12))$ on $\{1, \dotsc, 12 \}$, consider the action 
\begin{equation} \label{action}
\sigma \star (k_1, \dotsc, k_{12}) = (k_{\sigma(1)}, \dotsc, k_{\sigma(12)}).
\end{equation}
If $\sigma$ is invariant on $U(n)$, then so is the cyclic group $\langle \sigma \rangle$. Let $Y$ be a set of representatives of the orbits $\langle \sigma \rangle \ib{x}$, $\ib{x} \in U(n)$, under the action of $\langle \sigma \rangle$ by (\ref{action}). Thus, 
\begin{equation*}
U(n) = \bigsqcup_{\ib{y} \, \in \, Y}\langle\sigma\rangle \ib{y}.
\end{equation*}
Note that $\sigma$ is invariant on $U(n)$ exactly when it is invariant on 
$$V_p := \{(\ib{a}, \ib{b}, \ib{c}, \ib{d}) \in U(mp^r) \,: \, p \mid (\ib{a}, \ib{b}, \ib{c}, \ib{d})\}.$$ 
Also, if $p \mid (\ib{a}, \ib{b}, \ib{c}, \ib{d})$, then $(\ib{a}, \ib{b}, \ib{c}, \ib{d}) \in \langle\sigma\rangle \ib{x}$ exactly when $(\ib{a}, \ib{b}, \ib{c}, \ib{d})/p \in \langle\sigma\rangle (\ib{x}/p)$. Finally $(\ib{a}, \ib{b}, \ib{c}, \ib{d}) \in V_p$ if and only if $(\ib{a}, \ib{b}, \ib{c}, \ib{d})/p \in U(n/p)$ and so for a set of representatives $\mathcal{Y}$ of the orbits $\langle \sigma \rangle \ib{x}$, $\ib{x} \in V_{p}$, under the action of $\langle \sigma \rangle$
\begin{equation} \label{repe}
V_p = \bigsqcup_{\ib{y} \, \in \, \mathcal{Y}}\langle\sigma\rangle \ib{y} \,\,\, \text{if and only if} \,\,\, U(n/p) = \bigsqcup_{\ib{y} \, \in \, \mathcal{Y}}\langle\sigma\rangle (\ib{y}/p).
\end{equation}
Hence we can use the same set of representatives to index a partition of $V_p$ and of $U(n/p)$.
\begin{proposition} \label{third} For all primes $p \geq 3$ and integers $m,$ $r\geq1$
\begin{equation} \label{part1}
\sum_{\substack{(\ib{a}, \ib{b}, \ib{c}, \ib{d}) \, \in \, U(mp^r) \\ p \, 	\mid \, (\ib{a}, \ib{b}, \ib{c}, \ib{d})}} B'(\ib{a}, \ib{b}, \ib{c}, \ib{d}) \equiv \sum_{(\ib{a}, \ib{b}, \ib{c}, \ib{d}) \, \in \, U(mp^{r-1})} B'(\ib{a}, \ib{b}, \ib{c}, \ib{d}) \pmod{p^{3r}}
\end{equation}
and
\begin{equation} \label{part2}
\sum_{\substack{(\ib{a}, \ib{b}, \ib{c}, \ib{d}) \, \in \, U(mp^r) \\ p \, 	\mid \, (\ib{a}, \ib{b}, \ib{c}, \ib{d})}} B(\ib{a}, \ib{b}, \ib{c}, \ib{d}) \equiv \sum_{(\ib{a}, \ib{b}, \ib{c}, \ib{d}) \, \in \, U(mp^{r-1})} B(\ib{a}, \ib{b}, \ib{c}, \ib{d}) \pmod{p^{3r}}.
\end{equation}
For $p=2$, $(\ref{part1})$ holds modulo $2^{2r}$ while $(\ref{part2})$ holds modulo $2^{2r}$ for $r \geq 2$ and modulo $2$ otherwise.
\end{proposition}

\begin{proof}
For primes $p \geq 5$, (\ref{part1}) and (\ref{part2}) follow directly from Proposition \ref{second} and Corollary \ref{secondB}, respectively. For $p=3$, one can check that
the permutation $\sigma_1 =  (2,3,1,7,8,9,10,12,11,4,6,5)$ satisfies $B'(\sigma_1\star(\ib{a}, \ib{b}, \ib{c}, \ib{d})) = B'(\ib{a}, \ib{b}, \ib{c}, \ib{d})$. For a set of representatives $\mathcal{Y}_1$ of the orbits of $V_3$ under the action of $\langle \sigma_1 \rangle$, we have
\begin{equation} \label{simple}
\sum_{(\ib{a}, \ib{b}, \ib{c}, \ib{d}) \, \in \, V_3} B^{\prime}(\ib{a}, \ib{b}, \ib{c}, \ib{d}) = \sum_{\ib{y} \, \in \, \mathcal{Y}_1}\sum_{(\ib{a}, \ib{b}, \ib{c}, \ib{d}) \, \in \, \langle\sigma_1\rangle \ib{y}} B^{\prime}(\ib{a}, \ib{b}, \ib{c}, \ib{d}) = \sum_{\ib{y} \, \in \, \mathcal{Y}_1}|\langle\sigma_1\rangle \ib{y}| B^{\prime}(\ib{y}).
\end{equation}
As $|\sigma_1| = 3$ and $|\langle\sigma_1\rangle \ib{y}|$ divides $|\sigma_1|$, $|\langle\sigma_1\rangle \ib{y}| \in \{1,3\}$. If $|\langle\sigma_1\rangle \ib{y}| = 1$ then $\sigma_1\star(\ib{a}, \ib{b}, \ib{c}, \ib{d}) = (\ib{a}, \ib{b}, \ib{c}, \ib{d})$ and so $b_1=c_1=d_1$. Thus, $\binom{m 3^r}{b_1,c_1,d_1}$ is of the form $\binom{3a}{a,a,a}$ for some integer $a \geq 1$ and by Lemma \ref{div} we have $3 \mid \binom{m 3^r}{b_1,c_1,d_1}$. Expressing $B^{\prime}(\ib{a}, \ib{b}, \ib{c}, \ib{d})$ as in (\ref{form}) yields
\begin{equation*}
3\binom{m 3^r}{a_1} \binom{m 3^r}{a_2} \binom{m 3^r}{a_3} \mid B'(\ib{a}, \ib{b}, \ib{c}, \ib{d})
\end{equation*}
and so by (\ref{easier}) 
\begin{equation} \label{here}
B^{\prime}(\ib{a}, \ib{b}, \ib{c}, \ib{d}) \equiv 0 \pmod{3^{3r-s_1-2s_2+1}}.
\end{equation}
By the analogue of (\ref{ratio}) for $B^{\prime}(\ib{a}, \ib{b}, \ib{c}, \ib{d})$ when $p = 3$, where $3^{s_1 + 2s_2}$ is replaced by $3^{s_1 + 2s_2 - 1}$, and (\ref{here}), we have
\begin{equation*}
B'(\ib{a}, \ib{b}, \ib{c}, \ib{d}) \equiv B'((\ib{a}, \ib{b}, \ib{c}, \ib{d})/3) \pmod{3^{3r}}.
\end{equation*}
Thus, 
\begin{equation} \label{case1}
|\langle\sigma_1\rangle \ib{y}|B'(\ib{y}) \equiv |\langle\sigma_1\rangle (\ib{y}/3)|B'(\ib{y}/3) \pmod{3^{3r}}.
\end{equation}
If $|\langle\sigma_1\rangle \ib{y}| = 3$, then since $B'(\ib{y}) \equiv B'(\ib{y}/3) \pmod{3^{3r-1}}$ by Proposition \ref{second}, we again have (\ref{case1}).
By (\ref{simple}) and (\ref{case1}), 
\begin{align*}
\sum_{(\ib{a}, \ib{b}, \ib{c}, \ib{d}) \, \in \, V_3} B'(\ib{a}, \ib{b}, \ib{c}, \ib{d})  = \sum_{\ib{y} \, \in \, \mathcal{Y}_1}|\langle\sigma_1\rangle \ib{y}|B'(\ib{y})  &\equiv \sum_{\ib{y} \, \in \, \mathcal{Y}_1}|\langle\sigma_1\rangle (\ib{y}/3)|B'(\ib{y}/3) \pmod{3^{3r}}\\
&\equiv \sum_{\ib{y} \, \in \, \mathcal{Y}_1} \sum_{(\ib{a}, \ib{b}, \ib{c}, \ib{d}) \, \in \, \langle \sigma_1 \rangle (\ib{y}/3)} B'(\ib{a}, \ib{b}, \ib{c}, \ib{d}) \pmod{3^{3r}} \\
&=  \sum_{(\ib{a}, \ib{b}, \ib{c}, \ib{d}) \, \in \, U(m 3^{r-1})} B'(\ib{a}, \ib{b}, \ib{c}, \ib{d}) \pmod{3^{3r}}
\end{align*}
where the final line follows from (\ref{repe}). This proves (\ref{part1}) for $p=3$. The proof of (\ref{part2}) for $p=3$ is the same using Corollary \ref{secondB}. For $p=2$, the permutation $\sigma_2 = (6,4,5,9,8,7,12,10,11,3,2,1)$ satisfies $B'(\sigma_2\star(\ib{a}, \ib{b}, \ib{c}, \ib{d})) = B'(\ib{a}, \ib{b}, \ib{c}, \ib{d})$. Similar to (\ref{simple}),
\begin{equation} \label{simple1}
\sum_{(\ib{a}, \ib{b}, \ib{c}, \ib{d}) \, \in \, V_2} B'(\ib{a}, \ib{b}, \ib{c}, \ib{d}) = \sum_{\ib{y} \, \in \, \mathcal{Y}_2} \sum_{(\ib{a}, \ib{b}, \ib{c}, \ib{d}) \, \in \, \langle\sigma_2\rangle \ib{y}} B'(\ib{a}, \ib{b}, \ib{c}, \ib{d}) = \sum_{\ib{y} \, \in \, \mathcal{Y}_2} |\langle \sigma_2 \rangle \ib{y}| B'(\ib{y})
\end{equation}
where $\mathcal{Y}_2$ is a set of representatives of the orbits of $V_2$ under the action of $\langle \sigma_2 \rangle$. As $| \sigma_2 | = 4$ and $| \langle\sigma_2\rangle \ib{y} |$ divides $| \sigma_2 |$, $| \langle\sigma_2\rangle \ib{y} | \in \{1,2,4\}$. Proposition \ref{second} implies
\begin{equation*}
B'(\ib{a}, \ib{b}, \ib{c}, \ib{d}) \equiv B'((\ib{a}, \ib{b}, \ib{c}, \ib{d})/2) \pmod{2^{2r-1}}
\end{equation*}
and so if $|\langle\sigma_2\rangle \ib{y}| \in \{2,4\}$, it follows 
\begin{equation} \label{case3}
|\langle\sigma_2\rangle \ib{y}|B'(\ib{y}) \equiv |\langle\sigma_2\rangle (\ib{y}/2)|B'(\ib{y}/2) \pmod{2^{2r}}.
\end{equation}
When $|\langle\sigma_2\rangle \ib{y}| = 1$, more care is required. In this case, $\sigma_2 \star(\ib{a}, \ib{b}, \ib{c}, \ib{d}) = (\ib{a}, \ib{b}, \ib{c}, \ib{d})$ and so
\begin{equation*}
(a_1, a_2, a_3, b_1, b_2, b_3, c_1, c_2, c_3, d_1, d_2, d_3) = (a, b, c, b, c, a, a, c, b, c, b, a) := \ib{v}(a,b,c)
\end{equation*}
for some integers $a$, $b$, $c \geq 0$ with $a+b+c = m 2^r$. So, $B'(\ib{y}) = \binom{m 2^r}{a,b,c}^4$. There are now three cases. If $a=b=c$, then Lemma \ref{div} yields $2 \mid \binom{m 2^r}{a,a,a}$ and, as before, we have (\ref{case3}). If exactly two of $a$, $b$, $c$ are equal, say $a=b$, then 
\begin{equation*}
\binom{m 2^r}{a,a,b} = \binom{m 2^r}{b}\binom{m2^r-b}{a,a}
\end{equation*}
where $2a = m 2^r - b$. We now have two cases. If $a=0$, then $m 2^r = b$ and so 
\begin{equation} \label{a0}
B'(\ib{y}) = 1 = B'(\ib{y}/2). 
\end{equation}
Otherwise, if $a\geq1$, then $2 \mid \binom{m 2^r-b}{a,a}$ implies that $2 \mid \binom{mp^r}{a,a,b}$ and so, as before, (\ref{case3}) follows. Finally, if $a$, $b$ and $c$ are distinct, then each of the six permutations of $(a,b,c)$ in the argument of $\ib{v}(a,b,c)$ corresponds to a distinct element in $V_2$ whose orbit under the action of $\langle \sigma_2 \rangle$ is a singleton set and so correspond to distinct elements in $\mathcal{Y}_2$. In total, 
\begin{align*}
\sum_{\substack{\ib{y} \, \in \, \mathcal{Y}_2 \\ |\langle\sigma_2\rangle \ib{y}| \, = \, 1}} B'(\ib{y}) &= \sum_{\ib{v}(a,b,c) \, \in \, \mathcal{Y}_2} B'(\ib{y})  = \sum_{\substack{\ib{v}(a,b,c) \, \in \, \mathcal{Y}_2 \\ a,b,c \,\, \text{not all distinct}}} B'(\ib{y}) \,\, + \sum_{\substack{\ib{v}(a,b,c) \, \in \, \mathcal{Y}_2 \\ a, b, c\,\, \text{distinct}}} B'(\ib{y}) \\
&= \sum_{\substack{\ib{v}(a,b,c) \, \in \, \mathcal{Y}_2 \\ a,b,c \,\, \text{not all distinct}}} B'(\ib{y}) \,\, + 6\sum_{\substack{\ib{v}(a,b,c) \, \in \, \mathcal{Y}_2 \\ a \, > \, b \, > \, c}} B'(\ib{y}) \\
&= \sum_{\substack{\ib{v}(a,b,c) \, \in \, \mathcal{Y}_2 \\  a,b,c \,\, \text{not all distinct}}} B'(\ib{y}) \,\, + 3\sum_{\substack{\ib{v}(a,b,c) \, \in \, \mathcal{Y}_2 \\ a \, > \, b \, > \, c}} 2B'(\ib{y}) \\
&\equiv \sum_{\substack{\ib{v}(a,b,c) \, \in \, \mathcal{Y}_2 \\ a,b,c \,\, \text{not all distinct}}} B'(\ib{y}/2) \,\, + 3\sum_{\substack{\ib{v}(a,b,c) \, \in \, \mathcal{Y}_2 \\ a \, > \, b \, > \, c}} 2B'(\ib{y}/2) \pmod{2^{2r}} \\
&\equiv  \sum_{\substack{\ib{y} \, \in \, \mathcal{Y}_2 \\ |\langle\sigma_2\rangle (\ib{y}/2)| = 1}} B'(\ib{y}/2)  \pmod{2^{2r}}.
\end{align*}
Hence, by (\ref{simple1}) and (\ref{case3})
\begin{align*}
\sum_{(\ib{a}, \ib{b}, \ib{c}, \ib{d}) \, \in \, V_2} B'(\ib{a}, \ib{b}, \ib{c}, \ib{d}) &=  \sum_{\substack{\ib{y} \, \in \, \mathcal{Y}_2 \\ |\langle\sigma_2\rangle \ib{y}| \, \neq \, 1}} B'(\ib{y})  \,\, + \sum_{\substack{\ib{y} \, \in \, \mathcal{Y}_2 \\ |\langle\sigma_2\rangle \ib{y}| \, = \, 1}} B'(\ib{y}) \\
&\equiv \sum_{\substack{\ib{y} \, \in \, \mathcal{Y}_2 \\ |\langle\sigma_2\rangle (\ib{y}/2)| \, \neq \, 1}} B'(\ib{y}/2)  \,\, + \sum_{\substack{\ib{y} \, \in \, \mathcal{Y}_2 \\ |\langle\sigma_2\rangle (\ib{y}/2)| \, = \, 1}} B'(\ib{y}/2) \pmod{2^{2r}}\\
&\equiv \sum_{\ib{y} \, \in \,  \mathcal{Y}_2}\sum_{(\ib{a}, \ib{b}, \ib{c}, \ib{d}) \, \in \, \langle\sigma_2\rangle (\ib{y}/2)} B'(\ib{a}, \ib{b}, \ib{c}, \ib{d}) \pmod{2^{2r}} \\
&\equiv  \sum_{(\ib{a}, \ib{b}, \ib{c}, \ib{d}) \, \in \, U(m 2^{r-1})} B'(\ib{a}, \ib{b}, \ib{c}, \ib{d}) \pmod{2^{2r}}.
\end{align*}
This proves (\ref{part1}) for $p=2$. One can check that the proof of (\ref{part2}) for $p=2$ and $r \geq 2$ is similar using the fact that $\sigma_2$ satisfies $B(\sigma_2 \star (\ib{a}, \ib{b}, \ib{c}, \ib{d})) = B(\ib{a}, \ib{b}, \ib{c}, \ib{d})$, Corollary 4.5 and that (\ref{a0}) still holds in this case. Taking $r=1$ in (\ref{modpower2}) then completes the result.
\end{proof}

We can now prove Theorem \ref{main1}.

\begin{proof}[Proof of Theorem \ref{main1}]
By splitting the sum in (\ref{nosigndelta}) as in the beginning of the proof of Theorem 1.2 and applying Proposition \ref{third}, it suffices to prove 
\begin{equation} \label{last}
\sum_{\substack{(\ib{a}, \ib{b}, \ib{c}, \ib{d}) \, \in \, U(mp^r) \\ p \, \nmid \, (\ib{a}, \ib{b}, \ib{c}, \ib{d})}} B'(\ib{a}, \ib{b}, \ib{c}, \ib{d}) \,\, \equiv \sum_{\substack{(\ib{a}, \ib{b}, \ib{c}, \ib{d}) \, \in \, U(mp^r) \\ p \, \nmid \, (\ib{a}, \ib{b}, \ib{c}, \ib{d})}} B(\ib{a}, \ib{b}, \ib{c}, \ib{d}) \equiv 0 \pmod{p^{3r}}
\end{equation}
for all primes $p \geq 3$ and integers $m,$ $r\geq1$ while for $p = 2$ (\ref{last}) holds modulo $2^{2r}$. Following the proof of Theorem \ref{main}, we have $p^{2r}$ divides each summand in (\ref{last}) and so the result for $p=2$ is true. For $p \geq 3$ and the sum involving $B(\ib{a}, \ib{b}, \ib{c}, \ib{d})$ in (\ref{last}), one follows the argument as in the proof of Theorem \ref{main}. For $p \geq 5$ and the sum involving $B^{\prime}(\ib{a}, \ib{b}, \ib{c}, \ib{d})$, the only difference is that we need to replace (\ref{keyreduce}) with
\begin{equation*} \label{reduce4}
\sideset{}{'}\sum_{x=1}^{p^s - 1} \frac{1}{x^2} C(\ib{\ell}/p, mp^{r-1}, x/p) \equiv \sideset{}{'}\sum_{x=1}^{p^s - 1} \frac{1}{x^2} C(\ib{\ell}/p^{t}, mp^{r-t}, x/p^t) \pmod{p^s}.
\end{equation*}
By Lemma \ref{has}, each of (\ref{e1})--(\ref{e9}) holds where $(-1)^x$ is replaced with $1$ and so the result follows. Finally, for $p=3$ and the sum involving $B^{\prime}(\ib{a}, \ib{b}, \ib{c}, \ib{d})$, we recall that
\begin{equation*} \label{remain0}
\sum_{\substack{(\ib{a}, \ib{b}, \ib{c}, \ib{d}) \, \in \, U(m 3^r) \\ 3 \, \nmid \, (\ib{a}, \ib{b}, \ib{c}, \ib{d})}} B'(\ib{a}, \ib{b}, \ib{c}, \ib{d}) \equiv 6 \left (\sum_{\substack{(\ib{a}, \ib{b}, \ib{c}, \ib{d}) \, \in \, U_{\ib{a} \ib{b}}(m 3^r)}} B'(\ib{a}, \ib{b}, \ib{c}, \ib{d}) \right) \pmod{3^{3r}}
\end{equation*}
and so it is enough to demonstrate
\begin{equation*} \label{remain1}
2 \left (\sum_{\substack{(\ib{a}, \ib{b}, \ib{c}, \ib{d}) \, \in \, U_{\ib{a} \ib{b}}(m 3^r)}} 3B'(\ib{a}, \ib{b}, \ib{c}, \ib{d}) \right) \equiv 0 \pmod{3^{3r}}.
\end{equation*}
Similar to the proof of Theorem \ref{main}, it then suffices to prove
\begin{equation*} \label{reduce3}
\sideset{}{'}\sum_{x=1}^{3^s - 1} \frac{3}{x^2} C(\ib{\ell}/3, m 3^{r-1}, x/3) \equiv \sideset{}{'}\sum_{x=1}^{3^s - 1} \frac{3}{x^2} C(\ib{\ell}/3^{t}, m 3^{r-t}, x/3^t) \pmod{3^s}.
\end{equation*}
By Lemma \ref{has}, each of (\ref{e1})--(\ref{e9}) holds where $(-1)^x$ is replaced with $3$. Thus, the result for $p=3$ is true.
\end{proof}

\section*{Acknowledgements} 
The authors would like to again thank Armin Straub for his permission to include the proof of Proposition \ref{first} and for helpful comments and suggestions. The second author would like to thank the Okinawa Institute of Science and Technology for their hospitality and support during his visit from August 28 to September 15, 2023 as part of the TSVP Thematic Program ``Exact Asymptotics: From Fluid Dynamics to Quantum Geometry". The second author was partially funded by the Irish Research Council Advanced Laureate Award IRCLA/2023/1934. Finally, the authors are very grateful to the referee for their extremely careful reading of this paper.

\section*{Data availability}
Data sharing is not applicable to this article as no datasets were generated or analysed during the current study.

\end{document}